\documentclass{amsart}
\usepackage{amsmath, ifthen, amsthm, amssymb}

\usepackage{latexsym}
\usepackage{amsfonts}

\begin{document}

%%%%%%%macros%%%%%%%
%%%%%%%%%%%%%%%%%%

\newcommand{\n}{\ensuremath{n \in \omega}}
\newcommand{\om}{\ensuremath{\omega}}
\newcommand{\lh}{\mbox{\rm lh}}
\newcommand{\ep}{\ensuremath{\varepsilon}}
\newcommand{\set}[2]{\ensuremath{\{#1 \hspace{0.3mm} \mid \hspace{0.3mm} #2\}}}

%%%%to put tilde under bold%%%%
\newcommand\tboldsymbol[1]{%
\protect\raisebox{0pt}[0pt][0pt]{%
$\underset{\widetilde{}}{\boldsymbol{#1}}$}\mbox{\hskip 1pt}}
%%%%%%%%%%%%%%%%%%%%%

\newcommand{\del}{\ensuremath{\Delta^1_1}}
\newcommand{\sig}{\ensuremath{\Sigma^1_1}}
\newcommand{\pii}{\ensuremath{\Pi^1_1}}
\newcommand{\ca}[1]{\ensuremath{\mathcal{#1}}}
\newcommand{\W}{{\rm W}}
\newcommand{\WF}{{\rm WF}}
\newcommand{\ie}{{\rm i.e.,} }
\newcommand{\cf}{{\rm cf.~}}

\newcommand{\ck}{\ensuremath{\om_1^{\rm CK}}}
\newcommand{\ckr}[1]{\ensuremath{\om_1^{#1}}}

\newcommand{\LO}{{\rm LO}}
\newcommand{\WO}{{\rm WO}}

\newcommand{\hless}{\ensuremath{<_{\rm h}}}
\newcommand{\hleq}{\ensuremath{\leq_{\rm h}}}
\newcommand{\heq}{\ensuremath{=_{\rm h}}}

\newcommand{\hgeq}{\ensuremath{\geq_{\rm h}}}
\newcommand{\tgeq}{\ensuremath{\geq_{\rm T}}}

\newcommand{\M}{{\rm M}}
\newcommand{\baire}{\ca{N}}
\newcommand{\Baire}{\baire}
\newcommand{\cantor}{\ensuremath{2^\om}}
\newcommand{\cn}[2]{\ensuremath{#1 \ \hat{} \ #2}}

\newcommand{\St}{{\rm St}}
\newcommand{\tleq}{\leq_{\rm T}}

\newcommand{\rfn}[1]{\{#1\}}

\newcommand{\bolds}{\tboldsymbol{\Sigma}}
\newcommand{\boldp}{\tboldsymbol{\Pi}}
\newcommand{\boldd}{\tboldsymbol{\Delta}}

\newcommand{\tu}[1]{\textup{#1}}

\renewcommand{\ie}{\text{i.e.,}~}

%%%%%%%%%%%%%%%%%%%%%%%%%%%%%%%%%%%
%%%%%%%%%%%%%%%%%%%%%%%%%%%%%%%%%%%

\title[One more recursive-theoretic characterization of the TVC]{One more recursive-theoretic characterization of the Topological Vaught Conjecture}

\thanks{The author is currently a Fellow of the \textit{Programme 2020 researchers : Train to Move} at the Mathematics Department ``Guiseppe Peano" of the University of Turin, Italy. Some parts of this article were written, while the author was a Scientific Associate at TU Darmstadt, Germany.\\
Thanks are reserved to {\sc A. Montalban} for valuable suggestions.}

\address{Vassilios Gregoriades\\
Via Carlo Alberto, 10\\ 
10123 Turin, Italy}

\email{vassilios.gregoriades@unito.it}


\author[V. Gregoriades]{Vassilios Gregoriades}


\maketitle

\newtheorem{theorem}{Theorem} 
\newtheorem{lemma}[theorem]{Lemma}
\newtheorem{definition}[theorem]{Definition}
\newtheorem{proposition}[theorem]{Proposition}
\newtheorem{corollary}[theorem]{Corollary}
\newtheorem{remark}[theorem]{Remark}
\newtheorem{remarks}[theorem]{Remarks}
\newtheorem{example}[theorem]{Example}
\newtheorem{examples}[theorem]{Examples}
\newtheorem{claim}{Claim}

\begin{abstract}
We prove in {\bf ZF} a recursive-theoretic characterization of the Topological Vaught Conjecture by  revisiting the fact that orbits in Polish $G$-spaces are Borel sets.
\end{abstract}

\section{Introduction}

The question of characterizing the (Topological) Vaught Conjecture in terms of recursion theory has been investigated by Montalban \cite{montalban_a_computability_theoretic_equivalent_to_VC,montalban_analytics_equivalence_relations_satisfying_HYP_is_recursive}, where he provides the answer under some determinacy hypothesis. In this note we provide one more characterization of the topological version of the conjecture in a recursive-theoretic language, which is actually provable in the Zermelo-Fraenkel theory {\bf ZF}. Another distinctive aspect of our characterization is that, unlike \cite{montalban_analytics_equivalence_relations_satisfying_HYP_is_recursive}, it refers to the Polish group actions rather to the more general case of analytic equivalence relations, although one of the directions holds in the latter case. We do not know if the other direction is also true in the general case.

We will make substantial use of many results from \emph{effective} descriptive set theory, and we make no attempt in presenting any form of introduction to the latter. The standard textbook on the subject is \cite{yiannis_dst}. We recall however some central notions.

The natural numbers are identified with the first infinite ordinal $\om$ and the \emph{Baire space} $\om^\om$ is denoted by \baire. The \emph{Cantor space} is $\cantor$, \ie the set of all members of the Baire space with values in $2 \equiv \{0,1\}$. By $\om^{< \om}$ we mean the set of all finite sequences of elements of \om, including the empty one. The length $\lh(s)$ of a given $s \in \om^{< \om}$ is the unique number $n$ for which $s = (s_0,\dots,s_{n-1})$; as usual $\om^{n}$ the set of all $s \in \om^{< \om}$ with length $n$. We also fix a recursive injection $\langle \cdot \rangle: \om^{< \om} \to \om$. If $t = \langle s_0, \dots, s_{n-1} \rangle$ and $i < n$ we denote by $(t)_i$ the number $s_i$. If $t$ does not have the preceding form or $i \geq n$ we let $(t)_i$ be $0$. We fix once and for all the enumeration $(q_s)_{s \in \om}$ of all non-negative rational numbers, $q_s = (s)_0 \cdot ((s)_1+1)^{-1}$. Given $\alpha \in \baire$ and $n \in \om$ we denote by $\alpha \upharpoonright n$ the finite sequence $(\alpha(0), \dots, \alpha(n-1))$.

We say that a Polish space $\ca{X}$ is \emph{recursive} if there is a compatible metric $d$ on $\ca{X}$ such that $(\ca{X},d)$ is recursively presented \cf \cite[3B]{yiannis_dst}. 

A \emph{Polish} $G$-space is a triple $(\ca{X},G,\cdot)$ such that $\ca{X}$ is a Polish space, $G$ is a Polish group, and $\cdot : G \times \ca{X} \to \ca{X}$ is a continuous action on $\ca{X}$. By $E_G$ we always mean the induced orbit equivalence relation:
\[
x E_G y \iff (\exists g \in G)[ x = g \cdot y],
\]
where $x,y \in \ca{X}$. By $G \cdot x$ we mean the equivalence class or else the \emph{orbit} of $x$.

An equivalence relation $E$ on some Polish space \ca{X} has perfectly many classes if there is a non-empty perfect set $P \subseteq \ca{X}$ such that for all $x, y \in P$ with $x \neq y$ we have that $(x,y) \not \in E$.

The famous \emph{Topological Vaught Conjecture} states that for every Polish $G$-space $(\ca{X},G,\cdot)$ the orbit equivalence relation $E_G$ has either countably many or perfectly many classes.

A \emph{recursive Polish $G$-space} is a triple $(\ca{X},G, \cdot)$ with the following properties:
\begin{enumerate}
\item The sets \ca{X} and $G$ are recursive Polish spaces.
\item The set $G$ is a group and the function $(x,y) \mapsto xy^{-1}$ is recursive.
\item The function $\cdot: G \times \ca{X} \to \ca{X}$ is a group action, and is recursive;
\end{enumerate}
\cf \cite{becker_kechris_the_dst_of_Polish_group_actions}.

The preceding notions \emph{relativize} with respect to some parameter $\ep \in \baire$, e.g. we can talk about \emph{$\ep$-recursive} (or else \emph{recursive in $\ep$}) Polish $G$-spaces. In fact every Polish $G$-space is $\ep$-recursive for some suitable $\ep \in \cantor$.

The notion of a partial recursive function extends to recursive Polish spaces in a natural way \cf \cite[7B]{yiannis_dst}. We denote by $\rfn{e}^x$ the $e$-th partial $x$-recursive function on \om \ to \om, where $x$ belongs to some recursive Polish space \ca{X}. (The latter space should be clear from the context.)

By \ck \ we mean the least non-recursive ordinal and by $\ckr{x}$ the least non-$x$-recursive one. Given points $x,y$ in a recursive Polish space we write $x \hleq y$ for $x \in \del(y)$. The symbol $\tleq$ stands for Turing reducibility between members of $\cantor$.

\section{The characterization}

We can now state our recursive-theoretic characterization of the Topological Vaught Conjecture. For simplicity we state the result for recursive Polish $G$-spaces, but of course the analogous result holds also in the relativized case.

\begin{theorem}
\label{theorem characterization}
For every recursive Polish $G$-space $(\ca{X},G,\cdot)$ the following are equivalent.
\begin{enumerate}
\item The induced orbit equivalence relation $E_G$ does not have perfectly many classes.
\item For all $\alpha$ the set of orbits $\set{G \cdot x}{\ckr{(\alpha,x)} = \ckr{\alpha}}$ is countable.\footnote{The proof of the Silver Dichotomy Theorem is a standard application of the \emph{Gandy-Harrington} topology in order to produce perfectly many classes. In order to do so one utilizes the fact that the latter topology is Polish on all sets of the form \set{x}{\ckr{(\beta,x)}= \ckr{\beta}} \cf \cite{gao_invariant_dst}. Theorem \ref{theorem characterization} suggests that, in the case of Polish group actions, the preceding technique may fail to produce perfectly many classes, because the sets, on which the Gandy-Harrington topology is most useful, induce only countably many classes.}
\end{enumerate}\smallskip
In fact the implication $\tu{(}2\tu{)} \Longrightarrow \tu{(}1\tu{)}$ holds for arbitrary $\sig$ equivalence relations $E$ in recursive Polish spaces, \ie

if $\ca{X}$ is a recursive Polish space and $E$ is a $\sig$ equivalence relation on \ca{X}, for which the $\set{[x]_{E}}{\ckr{(\alpha,x)} = \ckr{\alpha}}$ is countable for all $\alpha \in \Baire$, then $E$ does not have perfectly many classes.
\end{theorem}

Before proceeding to the proof we find it useful to discuss some well-known facts. The countable linear orderings are encoded by members of the Baire space in a natural way. For every $\alpha \in \baire$ we define $\leq_\alpha \subseteq \om \times \om$ as follows
\begin{align*}
{\rm Field}(\alpha) \ &=  \set{n \in \om}{\alpha(\langle n,n \rangle) = 1}\\
n \leq_\alpha m \ &\iff n,m \in {\rm Field}(\alpha) \ \& \ \alpha (\langle n, m \rangle) = 1.
\end{align*}
The set $\LO$ of \emph{codes} of countable \emph{linear orderings} is defined as follows
\[
\LO = \set{\alpha \in \baire}{\leq_\alpha \ \text{is a linear ordering on} \ {\rm Field}(\alpha)}.
\]
The set of codes of \emph{well-orderings} is
\[
\WO = \set{\alpha \in \LO}{\leq_\alpha \ \text{is a well-ordering on} \ {\rm Field}(\alpha)}.
\]
We also put
\[
\M^\ca{Y}(x) \equiv \M(x) = \set{y \in \ca{Y}}{\ckr{(x,y)} = \ckr{x}},
\]
where $\ca{X}$, $\ca{Y}$ are recursive Polish spaces. It is well-known that $\M^\ca{Y}(x)$ is a Borel and $\sig(x)$ subset of $\ca{Y}$, see \cite{gao_invariant_dst}. As it was proved by Spector \cf \cite{spector_recursive_well-orderings} the set $\M^\ca{Y}(x)$ contains all points in $\del(x)$, and so from the Thomasson-Hinnman Theorem \cf \cite{thomason_the_forcing_method_and_the_upper_lattice_of_hyperdegrees,hinman_some_applications_of_forcing_to_hierarchy_problems_in_arithmetic} the set $\M^\ca{Y}(x)$ is comeager.

The \emph{hyperjump $\W^x$ of $x \in \ca{X}$} is defined by
\[
\W^x = \set{e \in \om}{\rfn{e}^x \ \text{is total and in}\ \WO}.
\]
It is a well-known result of Spector that $\ckr{(x,y)} = \ckr{x}$ if and only if $\W^x \not \hleq y$. So the second assertion of Theorem \ref{theorem characterization} essentially says that for all $\alpha$ the set of orbits in the hypercone with basis $\W^{\alpha}$ is co-countable.

\begin{remark}\normalfont
We do not know if the direct implication of Theorem \ref{theorem characterization} extends to all analytic equivalence relations. The standard example of an analytic equivalence relation with uncountably many but not perfectly many classes does satisfy the second assertion of Theorem \ref{theorem characterization}:

Given $x, y \in \LO$ we define
\begin{align*}
x E y 
\iff& [x, y \not \in \WO] \ \text{or}  \ [\text{$\leq_x$, $\leq_y$ are isomorphic}].
\end{align*}
Clearly the preceding $E$ is a $\sig$ equivalence relation on $\LO$ with uncountably many classes. Moreover it is not hard to verify that it does not have perfectly many classes, since for any non-empty perfect $P \subseteq \LO$ as in the definition of  ``perfectly many'' we would be able to find some non-empty perfect $P' \subseteq P$ with $P'  \subseteq \WO$. This would imply:
\[
y \in \WO \iff y \in \LO \ \& \ (\exists x \in P')[\text{$\leq_y$ embeds in $\leq_x$}].
\]
The latter would imply that $\WO$ is a $\sig$ set, a  contradiction.

Now given $\alpha \in \baire$ the set $\ckr{\alpha}$ is countable, so there is a sequence $(y^\alpha_n)_{\n}$ of $\alpha$-recursive well-orderings such that for each $\xi < \ckr{\alpha}$ it holds $\xi = |y^\alpha_n|$ for some \n. It follows easily that
\[
\set{[x]_E}{x \in \WO \ \& \ \ckr{(\alpha,x)} = \ckr{\alpha}} = \set{[y^\alpha_n]_E}{\n}.
\]
Since there is only one class $[z]_E$ for $z \not \in \WO$ it follows that the set $\set{[x]_E}{\ckr{(\alpha,x)} = \ckr{\alpha}}$ is countable for all $\alpha \in \baire$, and hence the equivalence relation $E$ satisfies the second assertion of Theorem \ref{theorem characterization} too.

One more example of a $\sig$ equivalence relation with uncountably many but not perfectly many classes is
\[
x F y \iff \ckr{x} = \ckr{y},
\]
where $x,y \in \cantor$. 

One way to see the latter is by applying the main result of \cite{montalban_analytics_equivalence_relations_satisfying_HYP_is_recursive}, see Remark \ref{remark Montalbans result} below. Given $\alpha \in \baire$, using again that $\ckr{\alpha}$ is  countable, we choose a sequence $(x_n)_{\n}$ in $\cantor$ such that 
\[
\set{\ckr{x}}{\ckr{x} \leq \ckr{\alpha}, \ x \in \cantor} = \set{\ckr{x_n}}{\n}.
\]
Then for every $x \in \cantor$ with $\ckr{(\alpha,x)} = \ckr{\alpha}$, we have in particular that $\ckr{x} \leq \ckr{\alpha}$ and so there is some $n$ such that $\ckr{x} = \ckr{x_n}$, \ie $x F x_n$. Thus the set of classes $\set{[x]_F}{\ckr{(\alpha,x)} = \ckr{\alpha}}$ is $\set{[x_n]_F}{\n}$.

Hence the relation $F$ satisfies the second assertion of Theorem \ref{theorem characterization} as well.
\end{remark}\bigskip

As it is well-known \cf \cite[Theorem 7.3.1]{becker_kechris_the_dst_of_Polish_group_actions} in a Polish $G$-space one can decompose the domain \ca{X} into an $\omega_1$-sequence $(A_\xi)_{\xi < \omega_1}$ of Borel sets such that each restriction $E_G \cap (A_\xi \times A_\xi)$ is a Borel set. Moreover the $A_\xi$'s can be chosen to be $E_G$-invariant. We provide the following related result.

\begin{proposition}
\label{proposition Polish group actions}
For every recursive Polish $G$-space $(\ca{X},G,\cdot)$ and for all $\alpha \in \Baire$ the sets $E_G \cap \left(\M(\alpha) \times \ca{X}\right)$ and $E_G \cap \left(\ca{X} \times \M^\ca{X}(\alpha)\right)$ are Borel.
\end{proposition}

The proof of the proposition above will be given in the sequel. This result is also related to a result of Sami. To explain this better, we set first
\[
\ckr{G \cdot x,\ep} = \min \set{\ckr{(g \cdot x,\ep)}}{g \in G}
\]
where $(\ca{X},G, \cdot)$ is \ep-recursive. Sami proved that in every \ep-recursive Polish $G$-space $(\ca{X},G,\cdot)$ it holds: (i) every orbit $G \cdot x$ is a $\boldp^0_{\ckr{G \cdot x,\ep} + 2}$ set; (ii) if there exists some $\xi < \om_1$ such that every orbit is $G \cdot x$ is a $\boldp^0_\xi$ set, then $E_G$ is Borel; and therefore (iii) if $\ckr{G \cdot x,\ep} = \ckr{\ep}$ for all $x \in \ca{X}$ then $E_G$ is a Borel equivalence relation. 

Proposition \ref{proposition Polish group actions} gives another proof of the preceding statement (iii). To see this assume that for all $x \in \ca{X}$ it holds $\ckr{G \cdot x,\ep} = \ckr{\ep}$, \ie there is some $z \in \M(\ep)$ such that $z E_G x$. Then for all $x,y \in \ca{X}$ we have
\begin{align*}
x E_G y \iff& (\forall z \in \M(\ep))[x E_G z \ \longrightarrow \ z E_G y]\\
\iff& (\forall z)[z \not \in \M(\ep) \ \vee \ (x,z) \not \in E_G \ \vee \ (z,y) \in E_G \cap \left(\M(\ep) \times \ca{X}\right)].
\end{align*}
Using Proposition \ref{proposition Polish group actions} it follows that $E_G$ is coanalytic and therefore it is moreover a Borel subet of $\ca{X} \times \ca{X}$.\smallskip

Our characterization is proved with the help of the preceding proposition.
  
\begin{proof}[\textit{Proof of Theorem \ref{theorem characterization}}]
For the left-to-right-hand direction, given $\alpha \in \baire$ we consider the restriction $F:=E_G \cap \left(\M^\ca{X}(\alpha) \times \M^\ca{X}(\alpha)\right)$. Then $F$ is a Borel equivalence relation on the Borel set $\M^\ca{X}(\alpha)$. 

If the conclusion were not true then $F$ would have uncountably many equivalence classes and so from Silver's Dichotomy \cite{silver_counting_the_number_of_equivalence_classes_coanalytic_equivalence_relations} there would be some non-empty perfect set $P \subseteq \M^\ca{X}(\alpha)$ such that for all $x,y \in P$ with $x \neq y$ it holds $(x,y) \not \in E_G$. In particular $E_G$ would have perfectly many classes, a contradiction.

For the converse direction, consider some $\alpha$-recursive injection $\pi: \cantor \rightarrowtail \ca{X}$. It is enough to show that for some $z \neq w$ in $\cantor$ we have that $\pi(z) E_G \pi(w)$. The set
\[
A = \set{z \in \cantor}{\ckr{(\alpha,\pi(z))} = \ckr{\alpha}}
\]
is easily a $\sig(\alpha)$ subset of $\cantor$. Moreover it contains all points in $\del(\alpha)$ and so $A$ is comeager.  In particular $A$ is an uncountable set. From our hypothesis the set of classes $B = \set{[x]_{E_G}}{\ckr{(\alpha,x)} = \ckr{\alpha}}$ is countable. Clearly the function $z \mapsto [\pi(z)]_{E_G}$ carries $A$ inside $B$. Since $B$ is a countable set and $A$ is an uncountable one, it follows that the latter function cannot be one-to-one on $A$, \ie there are $z \neq w$ in $A$ such that $[\pi(z)]_{E_G} = [\pi(w)]_{E_G}$. In other words $\pi(z) E_G \pi(w)$.

The proof of the latter direction when we have an arbitrary $\sig(\ep)$ equivalence relation $E$ in an $\ep$-recursive Polish space \ca{X} is exactly the same.
\end{proof}

\begin{remark}\normalfont
\label{remark Montalbans result}
(a) Montalban \cite{montalban_analytics_equivalence_relations_satisfying_HYP_is_recursive} proved that under the axiom of $\bolds^1_1$-determinacy an analytic equivalence relation $E$ on $\cantor$ does not have perfectly many classes exactly when there is some $\ep \in \baire$ such that for all $\alpha \tgeq \ep$ every $x \hleq \alpha$ is $E$-equivalent to some $y \tleq \alpha$ (the latter property is called \emph{HYP-is-recursive on a cone}). Moreover he showed that the converse direction is in fact provable in {\bf ZF}.

The argument that we used to prove the converse direction of Theorem \ref{theorem characterization}, provides also a somewhat shorter (although with less information) proof    of the converse direction of Montalban's preceding result. To see this assume that $E$ satisfies that HYP-is-recursive on the cone with basis $\ep$ and let $\pi: \cantor \to \cantor$ be an $\alpha$-recursive injection with $\alpha \tgeq \ep$ and $E$ is $\sig(\alpha)$. Consider the set
\[
\mathcal{A} = \set{z \in \cantor}{(\exists y \tleq \alpha)[(\pi(z),y) \in E]}.
\]
Clearly $\mathcal{A}$ is a $\sig(\alpha)$ subset of $\cantor$, and from our hypothesis it contains all points $z \in \cantor$ with $z \hleq \alpha$. Hence $\ca{A}$ is comeager. On the other hand $\ca{A} = \cup_e \ca{A}_e$, where
\[
\ca{A}_e = \set{z \in \cantor}{\rfn{e}^\alpha \ \text{is total and} \ (\pi(z),\rfn{e}^\alpha) \in E}.
\]
Hence for some $e$ the set $\ca{A}_e$ is non-meager and in particular it contains two distinct points $z \neq w$. We then have $\pi(z) \ E \ \rfn{e}^\alpha \ E \ \pi(w)$. Hence $E$ cannot have perfectly many classes.\smallskip

(b) It is clear from Montalban's characterization and Theorem \ref{theorem characterization} that for orbit equivalence relations the condition ``HYP-is-recursive on a cone'' implies (in {\bf ZF}) condition (b) of the latter theorem. It would be interesting to see if there is a \emph{direct} proof of this fact, which may also work for arbitrary analytic equivalence relations.
\end{remark}

\subsection*{Orbits are Borel sets.}
It is a known result of D. E. Miller \cite[Theorem $2'$]{miller_douglas_on_the_measurability_of_orbits_in_Borel_actions} that orbits of Borel actions of Polish groups are Borel sets. Sami's result (i) that we mentioned above is a refinement of the latter fact. Another such refinement is given by Becker \cite{becker_the_topological_TVC_and_minimal_counterexamples}, from where it follows that every orbit $G \cdot x$ in a recursive Polish $G$-space is a $\del(\W^x)$ set. Although not explicitly mentioned by Becker, the following fact is immediate from his proof.

\begin{proposition}[\cf \cite{becker_the_topological_TVC_and_minimal_counterexamples}]
\label{proposition hyperjump}
For every recursive Polish $G$-space $(\ca{X},G,\cdot)$ we have for all  $x,y \in \ca{X}$ that
\[
x E_G y \iff (\exists g \in \del(\W^{x},y))[y = g \cdot x].
\]
In particular every orbit $G \cdot x$ is a $\del(\W^{x})$ set.\footnote{Notice that from the Kleene Basis Theorem we can always find such a $g$ in $\del(\W^{(x,y)})$. The merit of this result is that we can relax the hyperjump on one of the variables. Moreover by combining these two type of refinements (Sami' s result  and Proposition \ref{proposition hyperjump}) with Louveau Separation \cite{louveau_a_separation_theorem_for_sigma_sets} it follows that every orbit $G \cdot x$ in a recursive Polish space is lightface $\Pi^0_{\ckr{G \cdot x}+2}(\beta)$ set, for some $\beta \hleq \W^{x}$.}
\end{proposition}

To see how the latter proposition follows from Becker's arguments, we go to the proof of \cite[Lemma 3.5]{becker_the_topological_TVC_and_minimal_counterexamples} and we notice that the unique $g \in G$, which satisfies that $g \in K$ and $g \cdot y = z$ is a $\del(T_H,y,z)$ point. This is because the set $K$ is $\del(T_H)$ (as Becker remarks in order to check this one has to review an earlier result of Dixmier \cite{dixmier_dual_et_quasi_dual_dune_algebre_de_Banach_involutive}) and therefore the preceding $g$ is a member of a $\del(T_H,y,z)$ singleton. Since $T_H$ is a $\sig(y)$ subset of the naturals we have that $T_H \tleq W^y$. Hence $g \in \del(W^y,z)$.

We find it useful to provide a more detailed sketch of the proof of Proposition \ref{proposition hyperjump}, where the effective arguments are somewhat easier to follow. But before we do this we show that Proposition \ref{proposition Polish group actions}, which we used to prove our characterization, follows from Proposition \ref{proposition hyperjump}. 

Spector \cf \cite{spector_recursive_well-orderings} proved that for all $x,y \in \cantor$ the following hold: (a) $\W^x \hleq y$ implies $\ckr{x} < \ckr{y}$; (b) $\ckr{x} < \ckr{y}$ and $x \hleq y$ implies $\W^x \hleq y$. It is then an easy corollary that if $\ckr{(x, y)} = \ckr{x}$ then $\W^{(x, y)} \hleq (\W^x,y)$, for all $x, y \in \cantor$. In particular if $\ckr{x} = \ck$ then $\W^x \hleq ( \W,x )$ for all $x \in \cantor$. To see this assume that $\ckr{(x, y)} = \ckr{x}$ and apply (b) of the preceding result of Spector with $x'= x \oplus y := (x(0),y(0),x(1),y(1),\dots)$ and $y'=\W^x \oplus y$.\footnote{Although these results were initially given for members of the Cantor space, they hold also in recursive Polish spaces, with very mild modifications. For example we exchange $x \oplus y$ with the pair $(x,y)$ and we consider the Polish space $\ca{X} \times \ca{X}$.} 
 
It is well-known that if $T$ is a perfect tree which is generic in the sense of Sacks forcing then $\W^T \hleq (\W,T)$. In particular the latter relation holds for almost all $T$. Summing up we have the following.

\begin{corollary}[see also 3.13 in \cite{sacks_higher_recursion_theory}]
\label{corollary from Spectors WA is hyp in B}
For every recursive Polish space \ca{X}, all $\alpha \in \baire$, and all $x \in \ca{X}$ with $\ckr{(x,\alpha)} = \ckr{\alpha}$ \tu{(}in particular for almost all $x \in \ca{X}$\tu{)} we have that
\[
\W^{(\alpha,x)} \hleq (\W^{\alpha},x).
\]
\end{corollary}

\begin{remark}\normalfont
These results give a short effective proof that analytic sets have the Baire property. Let us see how. Suppose that $P \subseteq \ca{X}$ is (without loss of generality) $\sig$ and that $F \subseteq \cantor \times \baire$ is $\Pi^0_1$ such that
\[
P(x) \iff (\exists \beta \in \baire)F(x,\beta).
\]
Using the Kleene Basis Theorem \cf \cite{kleene_arithmetical_predicates_and_function_quantifiers} (see also \cite[4E.8]{yiannis_dst}) and Corollary \ref{corollary from Spectors WA is hyp in B} we have that for all $x \in \M \equiv \M(\emptyset)$,
\begin{align*}
P(x) \iff& \ (\exists \beta \in \del(\W^{x}))F(x,\beta)\\
       \iff& \ (\exists \beta \in \del(\W,x))F(x,\beta).
\end{align*}
The latter shows that the set $P$ is computed by a $\pii(\W)$ relation on $\M$. Since $\M$ is Borel it follows that the set $P \cap \M$ is both analytic and coanalytic, and so from Souslin's Theorem it is also Borel. Moreover the set $P \setminus \left(P \cap \M\right)$ is meager because $\M$ is comeager.
\end{remark}

We can now explain how to derive Proposition \ref{proposition Polish group actions} from Proposition \ref{proposition hyperjump}. Given a recursive $(\ca{X},G,\cdot)$, by applying the  Corollary \ref{corollary from Spectors WA is hyp in B} as above, we can see round-robin style that for all $\alpha$ and all $(x,y) \in \M(\alpha) \times \ca{X}$,
\begin{align*}
x E_G y 
\iff& \ (\exists g \in \del(\W^{x},y))[y = g \cdot x]\\
\iff& \ (\exists g \in \del(\W^{(\alpha,x)},y))[y = g \cdot x]\\
\iff& \ (\exists g \in \del(\W^{\alpha},x,y))[y = g \cdot x].
\end{align*}

Hence the set $E_G \cap \left(\M(\alpha) \times \ca{X}\right)$ is defined by a $\pii(\W^{\alpha})$-formula, and is in particular a coanalytic subset of $\ca{X} \times \ca{X}$. Since it is evidently an analytic set as well, it follows from the Souslin Theorem that the set $E_G \cap \left(\M(\alpha) \times \ca{X}\right)$ is Borel. The result for $E_G \cap \left(\ca{X} \times \M(\alpha)\right)$ is proved similarly using the fact that $E_G$ is symmetric.\bigskip

We conclude with a more detailed sketch of the proof of Proposition \ref{proposition hyperjump}. We fix a recursive Polish $G$-space $(\ca{X},G,\cdot)$. Recall that the \emph{stabilizer} $G_x$ of $x \in \ca{X}$ is the set $\set{g \in G}{g \cdot x = x}$. Given a Polish space \ca{Y}, by $F(\ca{Y})$ we mean the set of all closed subsets of \ca{Y} with the Effros-Borel structure \cf \cite[12.C]{kechris_classical_dst}. A function $\delta: F(\ca{Y}) \to \ca{Y}$ is a \emph{choice function} if for all $\emptyset \neq F \in F(\ca{Y})$ we have that $\delta(F) \in F$.\smallskip

\textbf{Step 1.}  By easy calculations one can check that for every choice function $\delta: F(G) \to G$ it holds
\begin{align*}
y \in G \cdot x \iff& (\exists g)[\delta(g G_x) = g \ \& \ y = g \cdot x]\\\label{equation strong equivalence orbit x}
                    \iff& (\exists! g))[\delta(g G_x) = g \ \& \ y = g \cdot x],
\end{align*}
where $\exists !$ stands for ``there exists unique".\smallskip

\textbf{Step 2.} A \emph{Souslin scheme} on a Polish space \ca{Y} is any family $(U_s)_{s \in \om^{<\om}}$ of subsets of \ca{Y}. We say that a given Souslin scheme $(U_s)_{s \in \om^{<\om}}$ is \emph{good} if $U_\emptyset = \ca{Y}$, $\overline{U_{\cn{s}{i}}} \subseteq U_s$, $U_s = \cup_i U_{\cn{s}{i}}$ and diam$(U_s) \leq 2^{-\lh(s)}$ if $s \neq \emptyset$, where $\cn{s}{i} = (s_0,\dots,s_{n-1},i) \ \text{and} \ n = \text{the length of} \ s$.

The \emph{associated} function of a good Souslin scheme $\ca{U}: =( U_s)_s$ on $\ca{Y}$ is $f: \baire \to \ca{Y} : \{f(\alpha)\} = \cap_n U_{(\alpha(0),\dots, \alpha(n-1))}$. It is easy to verify that the associated function of a good Souslin scheme is continuous, and if moreover the good Souslin scheme consists of open sets, then the associated function is also open.

\begin{proposition}[Folklore?]
\label{proposition effective Souslin scheme}
Every recursive Polish space $\ca{Y}$ admits a good Souslin scheme, which has a recursive associated function. 
\end{proposition}

\begin{proof}
Let $d$ be a compatible metric for \ca{Y} and $\bar{r} = (r_j)_{j \in \om}$ a compatible recursive presentation. We denote by $N(k)$ be the open ball $B(r_{(k)_0},q_{(k)_1})$ with center $r_{(k)_0}$ and $d$-radius $q_{(k)_1}$, where $(q_i)_{i \in \om}$ is the enumeration of all non-negative rational numbers that we fixed in the introduction.

By dividing $d$ with $1+d$ we may assume that $d \leq 1$. We fix some $k_0 \in \om$ such that $N(k_0) = \ca{Y}$.

The idea is to write each $N(k)$ as the recursive union of some basic neighborhoods $N(m)$, $m \in I_k$, with $\overline{N(m)} \subseteq N(k)$ and radius$(N(m)) \leq 2^{-1}\cdot$radius$(N(k))$. Here ``recursive" means that the set $I(k,n) \iff n \in I_k$ is recursive.

Given $r_j \in N(k)$ then for any $t \in \om$ with $0 < q_t < q_{(k)_1} - d(r_j, r_{(k)_0})$ we have that $\overline{N(\langle j, t \rangle)} \subseteq N(k)$. We define 
\begin{align*}
I(k,m) \iff& \ d(r_{(m)_0},r_{(k)_0}) < q_{(k)_1} \ \& \ 0 < q_{(m)_1} < q_{(k)_1} - d(r_j, r_{(k)_0})\\
& \ \& \ q_{(m)_1} \leq 2^{-1} \cdot q_{(k)_1} ,
\end{align*}
so that $\overline{N(m)} \subseteq N(k)$ and radius$(N(m)) \leq 2^{-1} \cdot$radius$(N(k))$, when $I(k,m)$ holds. Clearly $I$ is a recursive set and each $k$-section $I_k$ of $I$ is non-empty, provided that $q_{(k)_1} > 0$ (and thus $N(k) \neq \emptyset$).

We prove round-robin style that
\begin{align}
\cup_{m \in I_k} N(m) = \cup_{m \in I_k} \overline{N(m)} = N(k) 
\end{align}
for all $k \in \om$.

The left-to-right inclusions are clear. Now suppose that $x \in N(k)$ and choose some $t \in \om$ such that $0 < 2 \cdot q_t < q_{(k)_1} - d(x,r_{(k)_0})$. We consider some $r_j \in B(x,q_t)$.  We then have
\begin{align*}
d(r_j,r_{(k)_0}) < d(r_j,r_{(k)_0}) + q_t \leq d(r_j,x) + d(x,r_{(k)_0}) + q_t < 2 \cdot q_t + d(x,r_{(k)_0}) < q_{(k)_1}.
\end{align*}
This shows that $d(r_j,r_{(k)_0}) < q_{(k)_1}$ and also that $q_t < q_{(k)_1} - d(r_j,r_{(k)_0})$. Moreover from the inequality $2 \cdot q_t + d(x,r_{(k)_0}) < q_{(k)_1}$ we obtain that $q_t < 2^{-1} \cdot q_{(k)_1}$. Hence $I(k,\langle j, t \rangle)$ holds. Moreover $d(r_j,x) < q_t$, \ie $x \in N(\langle j,t \rangle)$. This settles the inclusion $N(k) \subseteq \cup_{m \in I_k} N(m)$.

Finally we define recursively on $\lh(s)$ the Souslin scheme $(U_s)_{s}$ and the auxiliary function $\tau: \om^{<\om} \to \om$ as follows:\footnote{Formally we define partial functions, as our definition does not exclude a priori the possibility that $I_{\tau(s)}$ is the empty set and so $m_s$ is not defined. Of course, since we always have positive radii, the latter case never occurs, and therefore our functions are in fact total.}
\begin{align*}
(U_\emptyset,\tau(\emptyset)) =& \ (N(k_0),k_0) = (\ca{Y},k_0)\\
(U_{\cn{s}{m}}, \tau(\cn{s}{m})) =& \ 
\begin{cases}
(N(m),m), & \ \text{if} \ I(\tau(s),m),\\
(N(m_s),m_s),& \ \text{if $\neg I(\tau(s),m)$, where $m_s = \min I_{\tau(s)}$}.
\end{cases}
\end{align*}
By an easy induction on the length of $s$ one can see that $U_s = N(\tau(s)) \neq \emptyset$ for all $s$. It is also easy to verify that  $U_s = \cup_{m \in \om} U_{\cn{s}{m}} = \cup_{m \in \om} \overline{U_{\cn{s}{m}}}$, and that radius$(U_s) \leq 2^{-\lh(s)}$ for all $s$. Hence $(U_s)_s$ is a good Souslin scheme.

Finally we show that the associated function $f: \baire \to \ca{Y}$ is recursive. Since $f(\alpha) \in U_{\alpha \upharpoonright n} = N(\tau(\alpha \upharpoonright n))$ and the radius of the latter set is at most $2^{-n}$ we have that $d(f(\alpha), r_{(\tau(\alpha \upharpoonright n))_0}) \leq 2^{-n}$ for all $n$. It is then easy to verify that 
\begin{align*}
f(\alpha) \in N(m) 
\iff& \ d(f(\alpha),r_{(m)_0}) < q_{(m)_1}\\
\iff& \ (\exists n)[2^{-n} < q_{(m)_1} - d(r_{(m)_0},r_{(\tau(\alpha \upharpoonright n))_0})],
\end{align*}
for all $\alpha,m$. Hence $f$ is recursive.
\end{proof}

Now we fix a recursive Polish $G$-space $(\ca{X},G,\cdot)$ and a good Souslin scheme $(U_s)_s$ for $G$ with a recursive associated function $f: \baire \to G$.

For all non-empty $F \in F(G)$ we define the pruned tree $$T_F = \set{s \in \om^{<\om}}{F \cap U_s \neq \emptyset}$$ and we let $\alpha_F$ be the leftmost infinite branch of $T_F$. We also consider the choice function 
\[
\delta: F(G) \to G : F \mapsto f(\alpha_F).\footnote{Notice that $F(G)$ may not be a recursive Polish space. However this is not an obstacle, since  in the computations we can easily bypass any direct reference to $F(G)$.}\smallskip
\]
\textbf{Step 3.} There exists an arithmetical relation $A \subseteq \baire \times \ca{X} \times G \times \om^{<\om}$ such that
\[
g \cdot G_x \cap U_s \neq \emptyset \iff A(\W^x,x,g,s)
\]
for all $x,g,s$.

In order to prove this we check first that for all $x,g,s$ it holds
\begin{align}\label{equation lemma estimations of St}
(\exists h)[hx =x \ \& \ g \cdot h \in U_s] \iff& (\exists i,k)\big \{[(\forall j)[r_j \in N(G,k) \longrightarrow g \cdot r_j \in U_{\cn{s}{i}}]]\\ \nonumber
                   & \hspace*{20mm} \& \ (\exists h)[h \cdot x =x \ \& \ h \in N(G,k)] \big\}
\end{align}
where $(r_j)_{j \in \om}$ is the recursive presentation of $G$ and $N(G,k)$ is the $k$-th basic neighborhood of $G$ which comes from $(r_j)_{j \in \om}$.

The right-hand side of the preceding equivalence essentially says that there is a basic neighborhood $N(G,k)$ of $G$, which contains a member of the stabilizer of $x$, and is contained in a set of the form $g \cdot \overline{U_{\cn{s}{i}}} \subseteq g \cdot U_s$. The latter equivalence is proved using the property of the Souslin scheme being good, the density of $(r_j)_{j \in \om}$ and the continuity of the group action.

Having established (\ref{equation lemma estimations of St}) we observe that, using the Kleene Basis Theorem, the $h$ on the right-hand side of the latter equivalence can be chosen to be recursive in $\W^x$. Hence by taking the arithmetical set
\[
C(\alpha,x,k) \iff (\exists h \tleq \alpha)[h \cdot x =x \ \& \ h \in N(G,k)]]
\]
we conclude that
\[
g \cdot G_x \cap U_s \neq \emptyset 
\iff (\exists h)[hx =x \ \& \ g \cdot h \in U_s] 
\iff (\exists i, k)[D(g,k,s,i) \ \& \ C(\W^x,x,k)],
\]
where $D$ is defined according to (\ref{equation lemma estimations of St}). We then take 
\[
A(\alpha,x,g,s) \iff (\exists i, k)[D(g,k,s,i) \ \& \ C(\alpha,x,k)].
\]

\textbf{Step 4.} There exists an arithmetical relation $Q \subseteq \baire \times G \times \om$ such that
\[
\delta(g \cdot G_x) \in N(G,k) \iff Q(\W^x,x,g,k)
\]
for all $x, g, k$.

To see this, let $N(\baire,s)$ for $s \in \om^{<\om}$ be the usual $s$-th basic neighborhood of $\baire$. Since $f: \baire \to G$ is recursive we have that
\[
f(\alpha) \in N(G,k) \iff (\exists s)[\alpha \in N(\baire,s) \ \& \ R^\ast(s,k)]
\]
for some recursive $R^\ast \subseteq \om^{<\om} \times \om$. Then we can easily see that
\begin{align*}
\delta(g \cdot G_x)  \in N(G,k) \iff& \ f(\alpha_{g \cdot G_x}) \in N(G,k)\\ \label{equation lemma estimation Rtau A}
\iff& \ (\exists s)[\alpha_{g \cdot G_x} \in N(\baire,s) \ \& \ R^\ast(s,k)],
\end{align*}
where $\alpha_{g \cdot G_x}$ is as above the left-most infinite branch of the tree $T_F$ for $F = g \cdot G_x \in F(G)$.

Now we observe that
\[
\alpha_{g \cdot G_x} \in N(\baire,s)
\iff \ g \cdot G_x \cap U_s \neq \emptyset \ \& \ (\forall t \in \om^{\lh(s)})[t <_{\rm lex} s \longrightarrow g \cdot G_x \cap U_t = \emptyset]
\]
and using the arithmetical relation $A$ in the preceding step, it follows that there exists an arithmetical relation $Q \subseteq \baire \times G \times \om^{<\om}$ such that
\[
\alpha_{g \cdot G_x} \in N(\baire,s) \iff Q(\W^x,x,g,s)
\]
for all $x,g,i,j,s$. 

\smallskip

\textbf{Step 5.} For all $x,y \in \ca{X}$ the \tu{(}possibly empty\tu{)} set
\[
A_{x,y}:=\set{g \in G}{\delta(g \cdot G_x) = g \ \& \ y = g \cdot x}
\]
is arithmetical in $(\W^x,y)$. This is immediate from the key property of the set $Q$ in Step 4 and the fact that  $\Sigma^0_n(\W^x,x,y) = \Sigma^0_n(\W^x,y)$ for every $n \geq 1$. From Step 1, it follows that $A_{x,y}$ is at most a singleton, and therefore when it is non-empty its unique point is $\del(\W^x,y)$. Hence
\begin{align*}
y \in G \cdot x 
\iff& \ A_{x,y} \neq \emptyset\\
\iff& \ (\exists g \in \del(\W^x,y))[\delta(g \cdot G_x) = g \ \& \ y = g \cdot x]\\
\iff& \ (\exists g \in \del(\W^x,y))[ y = g \cdot x].
\end{align*}
This finished the sketch of the proof.

\end{document}